%% file: Burau_and_Polyhedra.tex
\documentclass{amsart}

\usepackage{setup}

\title{The Burau Representation and Shapes of Polyhedra}
\author{Ethan Dlugie}
\address{Department of Mathematics, University of California, Berkeley, CA 94720, USA}
\email{Dlugie.E@math.berkeley.edu}
\date{}

\begin{document}
\begin{abstract}
    We use a geometric approach to show that the reduced Burau representation specialized at roots of unity has another incarnation as the monodromy representation of a moduli space of Euclidean cone metrics on the sphere, as described by Thurston. Using the theory of orbifolds, we leverage this connection to identify the kernels of these specializations in some cases, partially addressing a conjecture of Squier. The 4-strand case is the last case where the faithfulness question for the Burau representation is unknown, a question that is related e.g. to the question of whether the Jones polynomial detects the unknot. Our results allow us to place the kernel of this representation in the intersection of several topologically natural subgroups of $B_4$.
\end{abstract}

\maketitle
    
\section{Introduction}\label{sec:Introduction}
    In this paper, we consider two representations of groups arising in low dimensional topology. First is the (reduced) Burau representation of braid groups $$\beta_n:B_n \to \operatorname{GL}_{n-1}(\mathbb {Z}[t^\pm])$$ that has been studied for almost a century \cite{Burau1935UberVerkettungen}. Second is a monodromy representation of punctured sphere mapping class groups coming from a geometric structure on the moduli space of Euclidean cone spheres, $$\rho_{\vec k} : \Mod(S_{0,m};\vec k) \to \operatorname{PU}(1,m-3),$$ as described by Thurston in \cite{Thurston1998ShapesSphere}. It has been found using algebraic techniques that these seemingly disparate representations are quite closely related in that the latter is, in a sense, a specialization of the former \cite{McMullen2013BraidTheory,Venkataramana2014MonodromyLine}. Our first theorem is a slight rephrasing of those results, which we will establish in this work via geometric means. See the beginning of \cref{sec:Cone metrics} for an introduction to the terminology of Euclidean cone metrics used in the following statement.
    
        \begin{theorem}[The Burau representation and polyhedra monodromy] \label{thm:commutative_diagram}
        Fix a choice of curvatures $\vec{k}$, which is to say a tuple of real numbers $\vec k = (k_1,\dotsc,k_m)$ with each $0 < k_i < 2\pi$ and $\sum_{i=1}^m k_i=4\pi$. Suppose further that $n$ of these curvatures are equal, say $k_1=\dotsb=k_n$ with $n \leq m-1$, and write $k_* \in (0,2\pi)$ for this common value. Set $q=\exp(i(\pi-k_*))$. Then the following diagram commutes
        \begin{center}
            \begin{tikzcd}
            B_n \arrow[d, "\iota_*"'] \arrow[rr, "\beta_n"] &                                      &[-20pt] \beta_n(B_n) \arrow[d, "\ev(-q)"] \arrow[r,phantom,"{\subset}"] & [-25pt] {\GL_{n-1}(\mathbb{Z}[t^\pm])}\\
            \Mod(S_{0,m};\vec k) \arrow[r, "\rho_{\vec k}"]       & {\operatorname{PU}(1,m-3)} \arrow[r,phantom,"{\subset}"] & \operatorname{PGL}_{m-2}(\mathbb C)    &                                      
            \end{tikzcd}
        \end{center}
        where $\beta_n$ is the $n$-strand (reduced) Burau representation, $\rho_{\vec k}$ is the monodromy representation of the moduli space of cone metrics, $\iota_*$ is the map on mapping class groups induced by an inclusion of an $n$-times marked disk $\iota:D_n \hookrightarrow S_{0,m}$, and $\ev(-q)$ is a slight alteration of an evaluation map to be defined in \cref{def:eval_of_ext}.
    \end{theorem}

    In the case where $m=n+1$ and $D_n$ is included into an $(n+1)$-times punctured sphere, the evaluation map mentioned in this theorem really is just an evaluation. This allows us to realize the ``specialized'' Burau representation $\beta(-q)$, where the formal variable $t$ is evaluated at a given unit complex number $-q$, as one of these polyhedral monodromy representations.
    \begin{corollary}\label{cor:commutative_diagram_for_specialization}
        Let $\vec k = (k_1,\dotsc,k_{n+1})$ be as in \cref{thm:commutative_diagram} with $k_*=k_1=\dotsb=k_n$. Write $q=\exp(i(\pi-k_*))$ Then the following diagram commutes
        \begin{center}
            \begin{tikzcd}
            B_n \arrow[d,  two heads, "\iota_*"] \arrow[rr, "\beta(-q)"] &                                      &[-20pt] \GL_{n-1}(\mathbb C) \arrow[d, two heads] \\
            \Mod(S_{0,n+1};\vec k) \arrow[r, "\rho_{\vec k}"]       & {\operatorname{PU}(1,n-2)} \arrow[r,phantom,"{\subset}"] & \operatorname{PGL}_{n-1}(\mathbb C)
            \end{tikzcd}
        \end{center}
        where $\Mod(S_{0,n+1};\vec k)$ is the subgroup of the mapping class group of the $(n+1)$-times punctured sphere that preserves the $(n+1)$-st point and may freely permute the other points.
        
        This yields a containment $\ker(\beta(-q)) \leq \ncl_{B_n}(\tilde S) \cdot \langle \tau_n \rangle$ where $\tilde S$ is a lift of a normal generatoring set for $\ker (\rho_{\vec k})$ and $\tau_n \in B_n$ is the full twist braid on $n$-strands that generates the center of the braid group.
    \end{corollary}

    In the statement, and in the rest of the paper, the notation $\ncl_G(S)$ indicates the normal closure of a set $S$ inside of a group $G$. We will also write $\tau_p \in B_n$ for a full twist about a curve surrounding $p$ points in the $n$-punctured disk. Any two such twists are conjugate in the braid group.
    
    In his influential paper \cite{Squier1984TheUnitary}, Squier briefly considered the specializations of the Burau representation at roots of unity. He made a conjecture about the form that the kernels of such specializations would take. We cannot verify Squier's conjecture in the form that he stated it,\footnote{See \cref{sec:conclusion} for a discussion on this point.} but using \cref{cor:commutative_diagram_for_specialization}, we are able to identify the kernel of these specializations in several cases.
    
    \begin{theorem}[Burau at roots of unity]\label{thm:Burau_roots_of_unity}
        Let $q$ be a primitive $d$-th root of unity and let $\beta(-q):B_n \to \GL_{n-1}(\mathbb C)$ denote the specialization of the Burau representation at $t=-q$. Then we have 
        \begin{equation}
            \ker (\beta(-q)) = \ncl_{B_n}(\sigma^d,\tau_{n-1}^j) \cdot \langle \tau_n^\ell \rangle \label{eqn:kernel_in_special_cases}
        \end{equation}
        for the following values of $n,d,j,l$:

        \begin{center}
        
            
            \begin{tabular}{c||cccccccc|cccc|cc|cc|c|c|c}
                $n$ & \multicolumn{8}{c|}{$4$}                                    & \multicolumn{4}{c|}{$5$}          & \multicolumn{2}{c|}{$6$} & \multicolumn{2}{c|}{$7$}  & $8$ & $9$ & $10$ \\ \hline
                $d$ & $5$ & $6$ & $7$ & $8$ & $9$ & $10$ & $12$ & $18$ & $4$ & $5$ & $6$ & $8$ & $4$ & $5$ & $3$ & $4$ & $3$ & $3$ & $3$  \\
                $j$ & $\infty$ & $\infty$ & $14$ & $8$ & $6$ & $5$  & $4$  & $3$  & $\infty$ & $5$ & $3$ & $2$ & $4$  & $2$ & $\infty$ & $2$ & $6$ & $3$ & $2$  \\
                $l$ & $5$ & $3$ & $7$ & $4$ & $9$ & $5$ & $3$ & $9$ & $4$ & $2$ & $3$ & $8$ & $2$ & $5$ & $6$ & $4$ & $3$ & $2$ & $3$ 
            \end{tabular}
        \end{center}

    Here $\sigma \in B_n$ denotes one of the half-twist generators of $B_n$ (all of which are conjugate), $\tau_{n-1} \in B_n$ denotes a full twist on a curve surrounding $n-1$ points in the punctured disk (all of which are conjugate), and $\tau_n \in B_n$ denotes the full twist on the boundary of the punctured disk (which generates the center of $B_n$). In a case with $j=\infty$, we mean that the kernel is $\ncl (\sigma^d) \cdot \langle \tau_n^\ell \rangle$ with no power of $\tau_{n-1}$.
    \end{theorem}

    We can also use the same method to identify the kernel of $\beta(-q)$ in all cases with $n=3$ and $d \geq 7$. The result is given in \cref{thm:Burau_3_at_roots_of_unity} and corrects the statement of \cite[Theorem 1.2]{Funar2014OnUnity}.

    Whether or not the Burau representation is faithful is a natural question to ask. At present, the answer is unknown only in the $n=4$ case, and this question has direct connections to the question of whether the Jones polynomial detects the unknot \cite{Bigelow2002DoesUnknot,Ito2015APolynomials}. An element of the kernel of $\beta_4$ must also lie in the kernel of every specialization. Thus \cref{thm:Burau_roots_of_unity} as a direct corollary restricts the kernel of $\beta_4$ to live in the intersection of several topologically natural normal subgroups of the braid group. One should note however that the intersection of these finitely many subgroups is still nontrivial by \cite[Lemma 2.1]{Long1986AGroups}, so this alone is not enough to establish faithfulness.
    
    \begin{corollary}[Narrowing $\ker (\beta_4)$] \label{cor:Burau_4_kernel}
        Let $\beta_4:B_4 \to \operatorname{GL}_3(\mathbb{Z}[t^\pm])$ denote the reduced Burau representation of the 4-strand braid group. Then
        \begin{align*}
            \ker (\beta_4 )& \leq \ncl_{B_4}(\sigma^d,\tau_3^\ell ) \cdot \langle \tau_4^\ell \rangle
        \end{align*}
        for powers $d,j,\ell$ as indicated in the table in \cref{thm:Burau_roots_of_unity}. All eight of these normal subgroups have infinite index in $B_4$.
    \end{corollary}
    
    In fact all of the normal subgroups of braid groups given by \cref{thm:Burau_roots_of_unity} have infinite index in their respective braid groups. I comment on the relationship between this and some remarkable work of Coxeter \cite{Coxeter1959FactorGroup} in \cref{sec:conclusion}.
    
    \subsection*{Some history and context.} The question of the faithfulness of the Burau representation has persisted since the representation was first defined nearly a century ago \cite{Burau1935UberVerkettungen}. Faithfulness is easily shown for $n=2,3$ (see e.g. \cite[Theorem 3.15]{Birman1975BraidsAM-82}). Faithfulness for other cases remained open for several decades. Squier put forth two conjectures \cite[(C1) and (C2)]{Squier1984TheUnitary} that, if both true, would yield the faithfulness of the Burau representation. However, Moody found the Burau representation to be nonfaithful for $n \geq 10$ \cite{Moody1993TheRepresentation}, and this result was quickly lowered to $n \geq 6$ by Long and Paton \cite{Long1993The6}. A few years later, Bigelow found a simpler example of an element in the kernel of $\beta_6$ and furthermore found that the Burau representation is not faithful for $n=5$ \cite{Bigelow1999The5}. Funar and Kohno proved Squier's conjecture (C2) in \cite{Funar2014OnUnity}, so we know for all $n \geq 5$ that Squier's conjecture (C1) is false for almost all (even) values of $d$. At the time of writing of this article, the faithfulness question is only open in the $n=4$ case.
    
    Braid groups are already known to be linear by another representation, the Lawrence-Krammer representation. See \cite{Krammer2002BraidLinear} for an algebraic treatment of this result and \cite{Bigelow2000BraidLinear} for a topological proof. Yet the faithfulness of the Burau representation, especially in the $n=4$ case, is still of interest due to its connection with the Jones polynomial in knot theory. Nonfaithfulness of $\beta_4$ implies that the Jones polynomial fails to detect the unknot \cite{Bigelow2002DoesUnknot,Ito2015APolynomials}. There has been work on the $n=4$ question in the last few decades. For instance, a computer search by Fullarton and Shadrach shows that a nontrivial element in the kernel of $\beta_4$ would have to be exceedingly complicated \cite{Fullarton2019ObservedRepresentation}, suggesting faithfulness. On the other hand, Cooper and Long found that $\beta_4$ is not faithful when taken with coefficients mod 2 and with coefficients mod 3 \cite{Cooper1997AmathbbZ_2,Cooper1998OnPrime}.
    
    Thurston's work in \cite{Thurston1998ShapesSphere} was a geometric reframing of the monodromy of hypergeometric functions considered by Deligne and Mostow in \cite{Deligne1986MonodromyMonodromy}. The algebro-geometric approach to studying these monodromy representations has continued, notably in works such as \cite{McMullen2013BraidTheory} and \cite{Venkataramana2014MonodromyLine}. The analysis of Euclidean cone metrics on surfaces was extended by Veech \cite{Veech1993FlatSurfaces} and is still today an active area of research in low-dimensional topology and dynamical systems.
    
    \subsection*{Organization of the paper.} The rest of the paper is organized as follows.
    \begin{itemize}
        \item \Cref{sec:Cone metrics} introduces Euclidean cone metrics on the sphere, and we construct explicit complex projective coordinates on the moduli space.
        \item In \cref{sec:Same Reps}, we prove \cref{thm:commutative_diagram} and \cref{cor:commutative_diagram_for_specialization} that allow us to relate the Burau representation at roots of unity with the monodromy representation of the moduli spaces of Euclidean cone metrics. Our proof uses the complex projective coordinates defined in \cref{sec:Cone metrics}.
        \item In \cref{sec:Completion of M} we gather several results about the complex hyperbolic geometry of the moduli space and facts about geometric orbifolds.
        \item In \cref{sec:narrow kernel} we prove \cref{thm:Burau_roots_of_unity} identifying the kernel of the Burau representation at some roots of unity. This uses \cref{cor:commutative_diagram_for_specialization} with the results of \cref{sec:Completion of M}. We also present the application of these ideas to the $\beta_3$ case in \cref{subsec:Burau_3}.
        \item \Cref{sec:conclusion} contains a discussion of limitations of this work and several possible future directions and connections that I hope can spark further research with these techniques.
    \end{itemize}
    
    \subsection*{Acknowledgements} Substantial thanks to my advisor, Ian Agol, who first informed me about the connection between the Burau representation and the polyhedra monodromy of Thurston's work. Special thanks as well to both Nancy Scherich and Sam Freedman for repeated helpful conversations through early drafts about the framing and phrasing of these results. Thanks to many others for taking the time to read and comment on a preprint of this work. Thanks to Louis Funar and Toshitake Kohno for graciously accepting a correction to one of their previous statements. And finally, I thank the anonymous referee for a speedy review of and many influential comments on my initial submission. In particular, the referee taught me about the conjectures of Squier which have come to be a focal point of this work in its current form.

    The author's work was supported in part by a grant from the Simons Foundation (Ian Agol, \#376200).
    
\section{Euclidean Cone Metrics on \texorpdfstring{$S^2$}{the Sphere}}\label{sec:Cone metrics}
    Here we recall the moduli space of Euclidean cone metrics on the sphere. We describe local coordinates on the moduli space into complex projective space. The construction is used in the proof of \cref{thm:commutative_diagram} in \cref{sec:Same Reps}.

    Following \cite{Thurston1998ShapesSphere}, we consider Euclidean cone metrics on the sphere. Such a metric is flat everywhere on the sphere away from some number of singular cone points $b_1,\dotsc,b_m$. Around each cone point $b_i$ one sees some cone angle not equal to the usual $2\pi$ that one finds around a smooth point. Define the \emph{curvature} $k_i$ at $b_i$ to be the angular defect of the cone point. The Gauss-Bonnet theorem applies with this notion of curvature to give $\sum_{i=1}^m k_i = 4\pi$.
    
    Thurston considers only those cone metrics which are nonnegatively curved, i.e. all $k_i > 0$.\footnote{A theorem of Alexandrov implies that every such metric arises uniquely as the intrinsic length metric on the boundary of a convex polyhedron in Euclidean space.} Fixing a tuple of positive real numbers $\vec k=(k_1,\dotsc,k_m)$ with each $0<k_i<2\pi$ and $\sum_{i=1}^m k_i=4\pi$, Thurston considers the moduli space of Euclidean cone metrics on the sphere with curvatures $\vec k$ up to orientation-preserving similarity. We denote this space $\mathcal{M}(\vec k)$.
    
    
    There is a natural map from $\mathcal{M}(\vec k)$ to (a finite cover of) the usual moduli space of conformal structures on the punctured sphere by simply taking the conformal class of a flat cone metric. There is also an inverse map inspired by the Schwarz-Christoffel mapping of complex analysis. Thurston uses this idea to show that the moduli space $\mathcal{M}(\vec k)$ is actually orbifold-isomorphic to the moduli space of conformal structures on $m$-punctured spheres with punctures labeled by the $k_i$ \cite[Proposition 8.1]{Thurston1998ShapesSphere}. This more classical moduli space is a complex orbifold of dimension $m-3$. The orbifold fundamental group of the moduli space is $\Mod(S_{0,m};\vec k)$, the group of mapping classes of the $m$-punctured sphere that preserve the labeling by curvatures.
    
    In his paper, Thurston shows directly that the moduli space of cone metrics has complex dimension $m-3$ by giving local $\mathbb{CP}^{m-3}$ coordinates on $\mathcal{M}(\vec k)$ in terms of cocyles on the sphere with twisted/local coefficients. Schwartz gave a more geometric interpretation of these coordinates in \cite{Schwartz2015NotesPolyhedra}. In brief, we have the following:
    
    \begin{lemma}[\cite{Schwartz2015NotesPolyhedra}] \label{lemma:moduli_coordinates}
        The moduli space $\mathcal{M}(\vec k)$ of Euclidean cone metrics on the sphere with $m$ cone points of curvatures $k_1,\dotsc, k_m$ is a complex projective orbifold of dimension $m-3$.
    \end{lemma}
    
    Taking a similar approach to Schwartz, we describe more concrete complex projective coordinates. This specific construction is for the sake of our calculation in the proof of \cref{thm:commutative_diagram}.
    
    \subsection{Local \texorpdfstring{$\mathbb{CP}^{m-3}$}{complex projective} coordinates}\label{subsec:coordinates}
    
    A $\mathbb{CP}^{m-3}$ structure on moduli space is encoded by a developing map $$\operatorname{dev}:\mathcal{T}(\vec k) \to \mathbb{CP}^{m-3}.$$ Here $\mathcal{T}(\vec k)$ is the Teichm\"uller space of Euclidean cone metrics on the sphere up to scaling with curvatures $\vec k$, which is the universal cover of the moduli space $\mathcal{M}(\vec k)$. We describe the map $\operatorname{dev}$ in a few steps.
    
    \begin{description}
        \item[A point in Teichm\"uller space] is a flat cone sphere $X$ and an isotopy class of homeomorphisms $f:S_{0,m} \to X$ from the $m$-times marked sphere $S_{0,m} = (S^2,\{b_1,\dotsc,b_m\})$ such that $f(b_i) \in X$ is a cone point of curvature $k_i$. The isotopy is taken relative to the cone points, and the metric on $X$ is only considered up to scaling.
        
        \item[A Euclidean developing] arises as follows. There is a disk $D \subset S_{0,m}$ such that $D \cap \{b_1,\dotsc,b_m\} = \partial D \cap \{b_1,\dotsc,b_m\} = \{b_1,\dotsc,b_{m-1}\}$, and $b_1$ through $b_{m-1}$ are arranged counterclockwise in order around $\partial D$. Abusing notation, we also write $b_1,\dotsc,b_m$ for the images of these points in $X$. Since $f(D) \subset X$ has no cone points in its interior and is simply connected, it has an isometric immersion (a developing map) to the Euclidean plane $d:f(D) \to \mathbb{E}^2$. This map is only defined up to orientation preserving planar isometries.
        
        \item[Complex coordinates] come from the positions of the cone points in $\mathbb{E}^2 \approx \mathbb{C}$, namely the tuple $(d(b_1),\dotsc,d(b_{m-1})) \in \mathbb{C}^{m-1}$. These points are well-defined up to isotopy of $f$, since the isotopy is relative to the cone points. If we instead record the differences $z_i=d(b_{i+1})-d(b_i)$, $i=1,\dotsc,m-2$, then we get a point $(z_1,\dotsc,z_{m-2}) \in \mathbb{C}^{m-2}$ that is well-defined up to a translation of $d$. The last coordinate $z_{m-1}=d(b_1)-d(b_{m-1})$ can be recovered from the rest. See \cref{fig:developing_cone_sphere}.
        
        \item[Finally, projectivization] is required because we can modify $d$ further by a rotation or real scaling ($X$ was only defined up to scaling). This means we have to mod out by the action of $\mathbb{C}^\times$, which is to say we get a well-defined point $[z_1:\dotsb:z_{m-2}] \in \mathbb{CP}^{m-3}$.
    \end{description}
    
    \begin{figure}
            \centering
            \def\svgwidth{\columnwidth}
            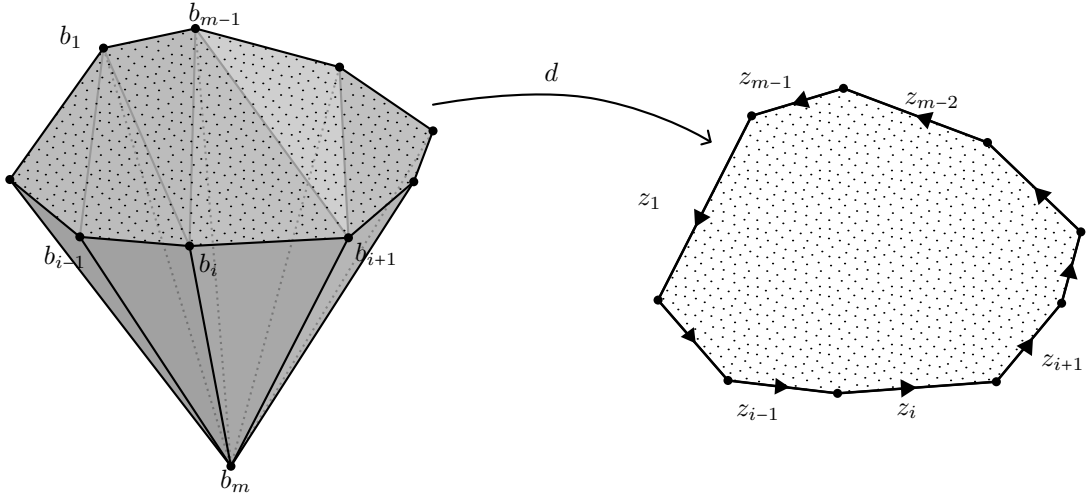
            \caption{A marked Euclidean cone metric on the sphere and the complex coordinates measured from the boundary of the developing image of a flat disk in the sphere.}
            \label{fig:developing_cone_sphere}
    \end{figure}
    
    All of the choices in the construction are accounted for, and we get a well defined assignment $(X,[f]) \mapsto [z_1:\dotsb:z_{m-2}]$. This defines $\operatorname{dev}:\mathcal{T}(\vec k) \to \mathbb{CP}^{m-3}$ and descends to a complex projective structure on moduli space.

\section{The Burau Representation and Polyhedra Monodromy}\label{sec:Same Reps}
    In this section we show that specializing the Burau representation to roots of unity yields (a portion of) the monodromy representation of Thurston's moduli space of Euclidean cone metrics.
    
    Considering cone metrics with curvatures $\vec k = (k_1,\dotsc,k_m)$, the mapping class group of interest is the subgroup of the $m$-punctured sphere mapping class group which preserves the labeling of points by their curvatures. Points labelled with the same curvature may be interchanged. We denote this group by $\Mod(S_{0,m};\vec k)$.
    
    If we have $k_1=\dotsb=k_n$ with $n \leq m$, then there is an action of the $n$-strand braid group $B_n \to \Mod(S_{0,m};\vec k)$. \Cref{thm:commutative_diagram} says that something akin to a specialization of the Burau representation appears in the monodromy representation of $\Mod(S_{0,m};\vec k)$ from Thurston's work. To connect the representations in the general case, which involves mapping between matrix groups of different dimensions, we need to define the evaluation map $\GL_{n-1}(\mathbb Z[t^\pm]) \to \PGL_{m-2}(\mathbb C)$ used in the statement of the theorem.

    \begin{definition}\label{def:affine_extension}
        Define a map $v:\GL_{n-1}(\mathbb Z[t^\pm]) \to (\mathbb Z(t))^{n-1}$ to the free module over the field of rational functions by sending a matrix $A$ to the column vector $$v(A) = \frac{1}{1-t^n} (I_{n-1}-A) \cdot \left( 1-t, 1-t^2,\dotsc,1-t^{n-1} \right)^\top.$$
        One can show that this map is a \textit{crossed homomorphism}.\footnote{The derivation of the formula for $v(A)$ comes from our understanding of the topological meaning behind the Burau representation. To keep the topic of this work contained, we simply present the formula here and comment on its utility in the course of the proof of \cref{lemma:Burau_lands_in_Laurent}.} We use it to define an \textit{affine extension} $\widetilde{(-)}:\GL_{n-1}(\mathbb Z[t^\pm]) \to \GL_n(\mathbb Z(t))$ by
        \begin{equation*}
            \widetilde A = \left( \begin{array}{c | c}
                A & v(A) \\ \hline 0 & 1
            \end{array} \right)
        \end{equation*}
        One can show that this map is a group homomorphism.
    \end{definition}

    We would like to evaluate this matrix at a complex number, but it might be possible that one of the entries of $v(A)$ is undefined at the desired evaluation. Conveniently, this is never the case for matrices in the image of the Burau representation.

    \begin{lemma}\label{lemma:Burau_lands_in_Laurent}
        For any $A \in \beta_n(B_n)$ one has $\widetilde{A} \in \GL_n(\mathbb Z[t^\pm])$, i.e. the matrix $\widetilde A$ has coefficients in the ring of Laurent polynomials rather than the field of rational functions.
    \end{lemma}
    \begin{proof}
        Recall (e.g. \cite[Chapter 3]{Birman1975BraidsAM-82}) that under the reduced Burau representation, the half-twist generator $\sigma_i$ acts as
        \begin{equation*}
            \beta_n(\sigma_i) = I_{i-2} \oplus \begin{pmatrix}
                1 & 0 & 0 \\
                t & -t & 1\\
                0 & 0 & 1
            \end{pmatrix}
            \oplus I_{n-i-2}
        \end{equation*}
        for $1 < i < n-1$ while
        \begin{equation*}
            \beta_n(\sigma_1) = \begin{pmatrix} -t & 1 \\ 0 & 1 \end{pmatrix} \oplus I_{n-3} \quad \text{and} \quad
            \beta_n(\sigma_{n-1}) = I_{n-3} \oplus \begin{pmatrix} 1 & 0 \\ t & -t \end{pmatrix}.
        \end{equation*}
        A straightforward computation gives that $v(\beta_n(\sigma_i)) = (0,\dotsc,0,0)$ for $i=1,\dotsc,n-2$, and $v(\beta_n(\sigma_{n-1})) = (0,\dotsc,0,1)$. We remark that the effect of the affine extension then is nothing more than to take the matrix $\beta_n(\sigma_i) \in \GL_{n-1}(\mathbb Z[t^\pm])$ to the matrix $\beta_{n+1}(\sigma_i) \in \GL_n(\mathbb Z[t^\pm])$.
        
        So the lemma holds for each $\beta_n(\sigma_i)$. Since $\GL_n(\mathbb Z[t^\pm])$ is a subgroup of $\GL_n(\mathbb Z(t))$ and the $\beta_n(\sigma_i)$ generate the image $\beta_n(B_n)$, the lemma follows.
    \end{proof}

    Since the matrices in the image of the Burau representation have Laurent polynomial coefficients, we can specialize the formal variable $t$ to any nonzero complex number and obtain a matrix with complex coefficients.

    \begin{definition}\label{def:eval_of_ext}
        Let $-q \in \mathbb C$ be a nonzero complex number and let $m \geq n+1$. Define the map $\ev(-q):\beta_n(B_n) \to \PGL_{m-2}(\mathbb C)$ by the composition
        \begin{equation*}
            \beta_n(B_n) \xrightarrow{\widetilde{(-)}} \GL_n(\mathbb Z[t^\pm]) \xrightarrow{(-) \oplus I_{m-2-n}} \GL_{m-2}(\mathbb Z[t^\pm]) \xrightarrow{t \mapsto -q} \GL_{m-2}(\mathbb C) \twoheadrightarrow \PGL_{m-2}(\mathbb C).
        \end{equation*}
        In the case $m=n+1$, we interpret the second map $(\text{--}) \oplus I_{-1}$ as deleting the last row and column of a matrix.
    \end{definition}

    \begin{remark}\label{rmk:factors_through_evaluation}
        In the case $m=n+1$, the second map serves to undo the effect of the first. Thus the map $\ev(-q)$ factors through the evaluation of the Burau representation:
        \begin{equation*}
            \beta_n(B_n) \subset \GL_{n-1}(\mathbb Z[t^\pm]) \xrightarrow{t \mapsto -q} \GL_{n-1}(\mathbb C) \twoheadrightarrow \PGL_{n-1}(\mathbb C).
        \end{equation*}
        This observation is the basis of the proof of \cref{cor:commutative_diagram_for_specialization}.
    \end{remark}

    Now we have the algebraic setup required to prove \cref{thm:commutative_diagram}.
    
    \begin{proof}[Proof of \cref{thm:commutative_diagram}]
        We prove the theorem by verifying that the diagram
        \begin{center}
            \begin{tikzcd}
                B_n \arrow[d, "\iota_*"'] \arrow[rr, "\beta_n"] &                                      &[-20pt] {\beta_n(B_n)} \arrow[d, "\ev(-q)"] \\
                \Mod(S_{0,m};\vec k) \arrow[r, "\rho_{\vec k}"]       & {\operatorname{PU}(1,m-3)} \arrow[r,phantom,"{\subset}"] & \operatorname{PGL}_{m-2}(\mathbb C)                                          
            \end{tikzcd}
        \end{center}
        commutes on a generating set for the braid group. In particular, we will use the $n-1$ Artin generators $\sigma_1,\dotsc,\sigma_{n-1} \in B_n \approx \Mod(D_n)$ that swap two points in the $n$-times marked disk with a counterclockwise half-twist.
        
        Label the marked points of $S_{0,m}$ as $b_1,\dotsc,b_m$ and consider cone metrics which have curvature $k_i$ at $b_i$ for each $i$. The $n$ points $b_1,\dotsc,b_n$ can be freely permuted by mapping classes since they have the same curvature $k_*=k_1 = \dotsb = k_n$, and so there is an inclusion of surfaces $\iota:D_n \hookrightarrow S_{0,m}$ sending the $n$ marked points of a marked disk $D_n$ to the points $b_1,\dotsc,b_n$. We write $\sigma_i$ for the generator of $B_n$, the counterclockwise half-twist on $D_n$ that swaps the points $\iota^{-1}(b_i)$ and $\iota^{-1}(b_{i+1})$.
        
        
        Now we compute the action of $\iota_*(\sigma_i)$ on the complex projective coordinates as constructed in \cref{subsec:coordinates}. Recall that our construction was to take a disk in a cone sphere $X$ on whose boundary lay the cone points $b_1,\dotsc,b_{m-1}$ and to look at the differences $z_1,\dotsc,z_{m-2}$ between those points under a developing map to $\mathbb C$. Applying a half twist $\iota_*(\sigma_i)$ changes that disk as indicated in the top of figure \cref{fig:polyhedra_computation}, and we consult the bottom of the figure for the effect on the developing image. The computation is done for one particular choice of homotopy class of disk in the cone sphere, but the induced action on the coordinates is independent of the chosen disk.
        
        The arc between $b_i$ and $b_{i+1}$ is taken to itself with reverse orientation on the cone sphere, but from within the disk it is approached from the opposite side, around the cone point $b_i$. Thus the developing image of this arc is rotated clockwise by the angle deficit/curvature at $b_i$, which is $k_*$, and is then reflected. This means $z_i$ is mapped to $-\exp(-ik_*)z_i = \exp(i(\pi-k_*))z_i$ under the monodromy representation of $\iota_*(\sigma_i)$. Similarly, $z_{i-1}$ maps to $z_{i-1}-\exp(i(\pi-k_*))z_i$, and $z_{i+1}$ maps to $z_i+z_{i+1}$. All other coordinates are fixed.
        
        \begin{figure}
            \centering
            \def\svgwidth{.7\columnwidth}
            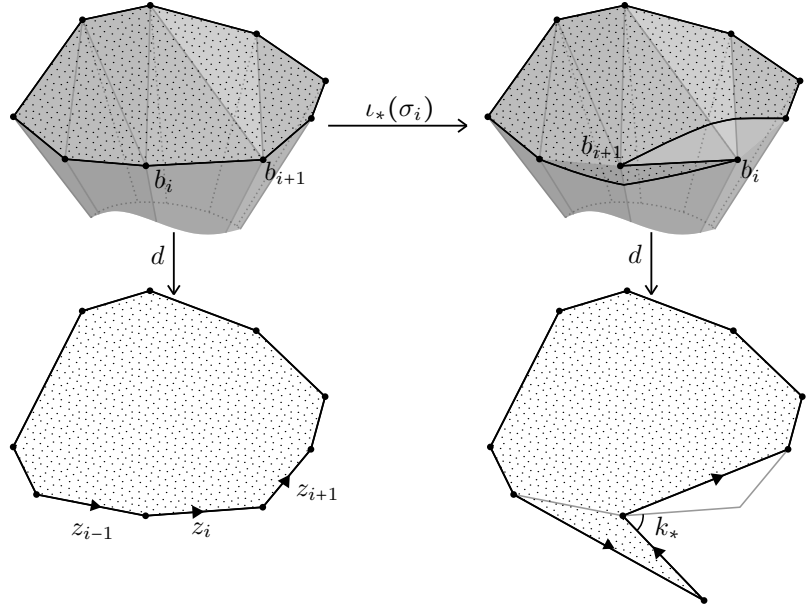
            \caption{The action of $\iota_*(\sigma_i)$ on the disk in $S_{0,m}$, indicated with polka dots, is depicted at top. At bottom is the effect on the developing image, with the outline of the original disk superimposed at right for comparison.}
            \label{fig:polyhedra_computation}
        \end{figure}
    
        So writing $q=\exp(i(\pi-k_*))$, we see that the coordinates $z_1,\dotsc,z_{m-2}$ are transformed linearly via the matrix
        \begin{equation*}
            \rho_{\vec k}(\iota_*(\sigma_i)) = I_{i-2} \oplus \begin{pmatrix}
                1 & 0 & 0 \\
                -q & q & 1\\
                0 & 0 & 1
            \end{pmatrix}
            \oplus I_{m-i-3}
        \end{equation*}
        for $1 < i < n-1$, with appropriate modifications for the cases $i=1$ and $i=n-1$ as necessary. Evidently this is the image of the Burau representation $\beta_n(\sigma_i)$ after our modified evaluation map $\ev(-q)$. Since the diagram commutes when applying the various maps to any $\sigma_i \in B_n$ and the $\sigma_i$ generate the braid group, we have the desired result.
    \end{proof}

    In the case $m=n+1$, the evaluation map $\ev(-q)$ is truly an evaluation of the Burau representation (rather than the evaluation of the affine extension of the representation). This is the basis for \cref{cor:commutative_diagram_for_specialization}.

    \begin{proof}[Proof of \cref{cor:commutative_diagram_for_specialization}]
        In the case $m=n+1$, \cref{rmk:factors_through_evaluation} tells us that the commutative diagram in \cref{thm:commutative_diagram} factors as
        \begin{center}
            \begin{tikzcd}
            B_n \arrow[dd, "\iota_*"'] \arrow[r, "\beta_n"] \arrow[rrd, "\beta(-q)"']   &   \beta_n(B_n) \arrow[r, "\subset", phantom]  &   [-20pt]{\GL_{n-1}(\mathbb Z[t^\pm])} \arrow[d, "t \mapsto -q"] \\
            &       &    \GL_{n-1}(\mathbb C) \arrow[d, two heads]        \\
            {\Mod(S_{0,n+1};\vec k)} \arrow[r, "\rho_{\vec k}"]                                       & {\operatorname{PU}(1,n-2)} \arrow[r, "\subset", phantom] & \operatorname{PGL}_{n-1}(\mathbb C)
        \end{tikzcd}
        \end{center}
        which gives the commutative diagram in the statement of the corollary. Thus an element of $\ker (\beta(-q))$ must lie in the kernel of the composite map
        \begin{equation*}
            B_n \xrightarrow{\iota_*} \Mod(S_{0,n+1};\vec k) \xrightarrow{\rho_{\vec k}} \PU(1,n-2). 
        \end{equation*}
        The inclusion of the punctured disk $\iota:D_n \hookrightarrow S_{0,n+1}$ induces a surjective map of mapping class groups $\iota_*:B_n \twoheadrightarrow \Mod(S_{0,n+1};\vec k)$ whose kernel is the central subgroup $\langle \tau_n \rangle$ generated by the full twist braid. (see e.g. \cite[Proposition 3.19]{Farb2012AGroups}). If $\ker (\rho_{\vec k}) = \ncl_{\Mod(S_{0,n+1};\vec k)}(S)$, then a product of conjugates of elements of $S \subset \Mod(S_{0,n+1};\vec k)$ lifts to a product of conjugates of elements of $\tilde S \subset B_n$.  Lifts differ by the kernel of $\iota_*$, so the kernel of the composition $\rho_{\vec k} \circ \iota_*$ equals $\ncl_{B_n}(\tilde S,\tau_n) = \ncl_{B_n}(\tilde S) \cdot \langle \tau_n \rangle$, the last equality following because the full twist $\tau_n$ is central in the braid group.
    \end{proof}
    
\section{The Completion of Moduli Space and Orbifolds}\label{sec:Completion of M}
    In this section we gather several results that we will leverage in combination with \cref{cor:commutative_diagram_for_specialization} to restrict the kernel of the Burau representation. First, we recall Thurston's results about the metric completion of the moduli space of Euclidean cone metrics. Then we gather facts about orbifold fundamental groups and the monodromy representations of geometric orbifolds. Orbifold structures will allow us to exactly identify the kernels of the representations $\rho_{\vec k}$.

    \Cref{lemma:moduli_coordinates} says that $\mathcal{M}(\vec k)$ is a complex projective orbifold, yielding a monodromy representation $\rho_{\vec k}:\Mod(S_{0,m};\vec k) \to \PGL_{m-2}(\mathbb C)$. But in fact there is an extra piece of information we have, a hermitian form on the complex coordinate space whose diagonal part $A:\mathbb{C}^{m-2} \to \mathbb R$ simply measures the area of a Euclidean cone sphere whose developing map (before scaling) yields the coordinates $z_1,\dotsc,z_{m-2} \in \mathbb C$. Since acting by mapping classes only changes the choice of disk in a Euclidean cone sphere $X$ and not the underlying geometry of $X$, we see that our monodromy representation lands in the unitary group of the area form, i.e. we have $\rho_{\vec k}:\Mod(S_{0,m}; 
    \vec k) \to \operatorname{PU}(A)$. Thurston gave a geometric method to find the signature of this form:
    
    \begin{lemma}[{\cite[Proposition 3.3]{Thurston1998ShapesSphere}}] \label{lemma:hyperbolic_signature}
        The area form $A$ has signature $(1,m-3)$.
    \end{lemma}
    
    And so $\operatorname{PU}(A) \approx \operatorname{PU}(1,m-3)$ is the group of holomorphic isometries of complex hyperbolic space $\mathbb{CH}^{m-3}$. Therefore $\mathcal{M}(\vec k)$ is in fact a \emph{complex hyperbolic} orbifold. As Thurston explains, the complex hyperbolic metric is not complete. The metric completion $\overline{\mathcal M}(\vec k)$ is obtained by adding several strata: lower-dimensional moduli spaces corresponding to the collision of groups of cone points whose curvatures sum to less than $2\pi$. Around each (real) codimension-2 stratum where two cone points collide there is a cone angle in the complex hyperbolic metric. Thurston examines the geometry of these strata.
    
    \begin{lemma}[{\cite[Proposition 3.5]{Thurston1998ShapesSphere}}] \label{lemma:cone_angles}
        The cone angle around a codimension-2 stratum of $\overline{\mathcal{M}}(\vec k)$ where cone points of curvature $k_1$ and $k_2$ collide is
        \begin{enumerate}
            \item $\pi-k_*$ when $k_*=k_1=k_2$, and
            \item $2\pi - (k_1+k_2)$ when $k_1 \ne k_2$.
        \end{enumerate}
    \end{lemma}
    
    In general, the completion of the moduli space of Euclidean cone metrics has the structure of a complex hyperbolic \textit{cone manifold}. We won't go into the particulars of cone manifolds, but we note that in certain cases the completed moduli space is actually a complex hyperbolic \textit{orbifold}.
    
    \begin{lemma}[{\cite[Theorem 4.1]{Thurston1998ShapesSphere}}] \label{lemma:orbifold_conditions}
        If the cone angles around all codimension-2 strata of $\overline{\mathcal M}(\vec k)$ are integral submultiples of $2\pi$, then the metric completion $\overline{\mathcal M}(\vec k)$ is a complex hyperbolic orbifold. A codimension-2 stratum with cone angle $2\pi/d$ is an orbifold stratum of order $d$.
    \end{lemma}

    \Cref{cor:commutative_diagram_for_specialization} allows us to understand the kernel of the specialized Burau representation because the monodromy representation $\rho_{\vec k}$ has a kernel that is quite explicit thanks to orbifold theory..
    
    First, the proof of \cite[Theorem 2.9]{Cooper2000Three-dimensionalCone-Manifolds} and the paragraph following it tell us how to find the orbifold fundamental group of the completed moduli space when this space has an orbifold structure:
    \begin{lemma}[{\cite[Theorem 2.9]{Cooper2000Three-dimensionalCone-Manifolds}}]\label{lemma:completion_kernel}
        Suppose the completed moduli space $\overline{\mathcal{M}}(\vec k)$ is a complex hyperbolic orbifold. Then the kernel of the map
        $$\Mod(S_{0,m};\vec k) = \pi_1^{orb}(\mathcal{M}(\vec k)) \twoheadrightarrow \pi_1^{orb}(\overline{\mathcal{M}}(\vec k))$$
        is the normal closure of powers of loops around the added codimension-2 strata of $\overline{\mathcal M}(\vec k)$, the power being the order of the given stratum.
    \end{lemma}

    And then as a special case of \cite[Theorem 2.26]{Cooper2000Three-dimensionalCone-Manifolds}, we have the following:
    \begin{lemma}[{\cite[Theorem 2.26]{Cooper2000Three-dimensionalCone-Manifolds}}]\label{lemma:orbifolds_faithful}
        Suppose the completed moduli space $\overline{\mathcal{M}}(\vec k)$ is a complex hyperbolic orbifold. Then the monodromy representation of the completed moduli space $\pi_1^{orb}(\overline{\mathcal{M}}(\vec k)) \to \PU(1,m-3)$ is faithful.
    \end{lemma}
    
\section{The Kernel of Burau at Roots of Unity}\label{sec:narrow kernel}
    In this section we explain how \cref{cor:commutative_diagram_for_specialization} and the results gathered in \cref{sec:Completion of M} can be used to identify the kernel of the Burau representation specialized at certain roots of unity. In particular, we will prove \cref{thm:Burau_roots_of_unity} and \cref{cor:Burau_4_kernel}.
    
    Our analysis uses the 94 choices of curvatures $\vec k$ enumerated by Thurston \cite{Thurston1998ShapesSphere} for which the completion of the moduli space of Euclidean cone metrics $\overline{\mathcal{M}}(\vec k)$ is a complex hyperbolic orbifold. In these cases, we know exactly the kernel of the monodromy representation $\rho_{\vec k}:\Mod(S_{0,m};\vec k) \to \operatorname{PU}(1,m-3)$. Searching Thurston's list of orbifolds allows us, via the commutative diagram of \cref{cor:commutative_diagram_for_specialization}, to restrict the kernel of the associated specialization of the Burau representation. First we demonstrate how our method applies in the $n=6$ case and verify it with a certain braid already known to lie in the kernel of $\beta_6$.

    \begin{example}[The case of $n=6$ and $d=4$] \label{ex:Burau_6}
        To obtain a containment on the kernel of the Burau representation, we appeal to a particular moduli space of cone structures. Consider the choice of curvatures $\vec k = 2\pi(1,1,1,1,1,1,2)/4$. The completion of the moduli space $\overline{\mathcal M}(\vec k)$ is formed by adding two codimension-2 stratum. One stratum corresponds to the collision of two cone points of curvature $2\pi/4$, and by \cref{lemma:cone_angles}, the cone angle around this stratum is $\pi - 2\pi(1/4)=2\pi/4$. The other stratum corresponds to the collision of a cone point of curvature $2\pi/4$ with the one of curvature $2\pi(2/4)$, and the cone angle here is again $2\pi-2\pi(1/4+2/4)=2\pi/4$. Since these angles are submultiples of $2\pi$, \cref{lemma:orbifold_conditions} says that $\overline{\mathcal M}(\vec k)$ is an orbifold. The added strata are orbifold strata both of order $4$.
        
        Now pick two mapping classes: $\sigma \in \Mod(S_{0,7};\vec k)$ that exchanges two cone points of equal curvature with a half-twist, and $\tau \in \Mod(S_{0,7};\vec k)$ that performs a full twist of a pair of points of distinct curvatures. Then $\sigma$ and $\tau$ represent the $\pi_1^{orb}$-conjugacy classes of loops around the added orbifold strata. \Cref{lemma:completion_kernel,lemma:orbifolds_faithful} give that the monodromy representation $\rho_{\vec k}:\Mod(S_{0,7};\vec k) \to \operatorname{PU}(1,4)$ has kernel exactly the normal subgroup $\ncl_{\Mod(S_{0,7};\vec k)}(\sigma^4,\tau^4)$. The evident inclusion of a 6-punctured disk $\iota:D_6 \hookrightarrow S_{0,7}$ induces a surjective map of mapping class groups $\iota_*:B_6 \twoheadrightarrow \Mod(S_{0,7};\vec k)$ whose kernel is the central subgroup $\langle \tau_6 \rangle$. (see e.g. \cite[Proposition 3.19]{Farb2012AGroups}).
        
        Conflating notation, the mapping classes $\sigma,\tau \in \Mod(S_{0,7};\vec k)$ lift to a half twist generator $\sigma \in B_6$ and a full twist on 5 strands $\tau_5 \in B_6$.
        \begin{figure}
            \centering
                \def\svgwidth{0.25\columnwidth}
                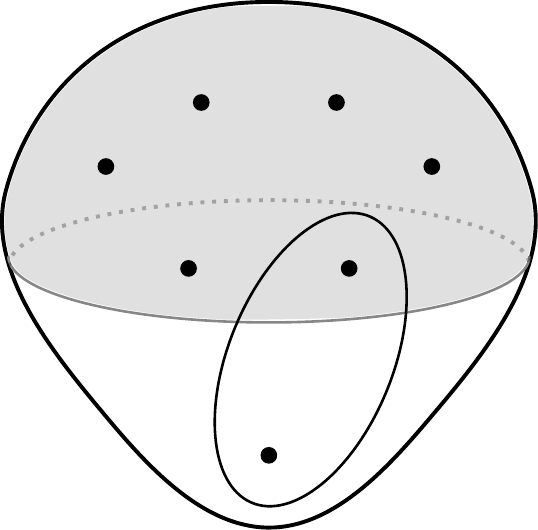
                \caption{The same curve encircles either one puncture in the disk and the puncture outside of the disk or $n-1$ punctures in the disk. So the full twist $\tau \in \Mod(S_{0,n+1};\vec k$) about this curve lifts to the twist $\tau_{n-1} \in B_n$ about $n-1$ points.}
                \label{fig:the_last_twist_in_the_disk}
        \end{figure}
        See \cref{fig:the_last_twist_in_the_disk}. So by \cref{cor:commutative_diagram_for_specialization} we have
        \begin{equation}
            \ker (\beta(-q)) \leq \ncl_{B_6}(\sigma^4,\tau_5^4) \cdot \langle \tau_6 \rangle. \label{eqn:burau_6_containment}
        \end{equation}
        for $q=\exp(2\pi i/4)$ a primitive 4-th root of unity.

        Any braid in the kernel of the Burau representation $\beta_6$, before specialization, must also lie in the kernel of any specialization. So \eqref{eqn:burau_6_containment} also gives a containment on $\ker( \beta_6)$. There is in fact a not-too-complicated, nontrivial braid known to lie in the kernel of $\beta_6$, found by Bigelow \cite{Bigelow1999The5}. The containment \eqref{eqn:burau_6_containment} implies that this braid is trivial if we allow ourselves the extra relations that $\sigma^4=\tau_5^4=\tau_6=1$. See \cref{fig:commutator_reduction} for a computation verifying this conclusion. \qed
        
        \begin{figure}
                \centering
                \def\svgwidth{\columnwidth}
                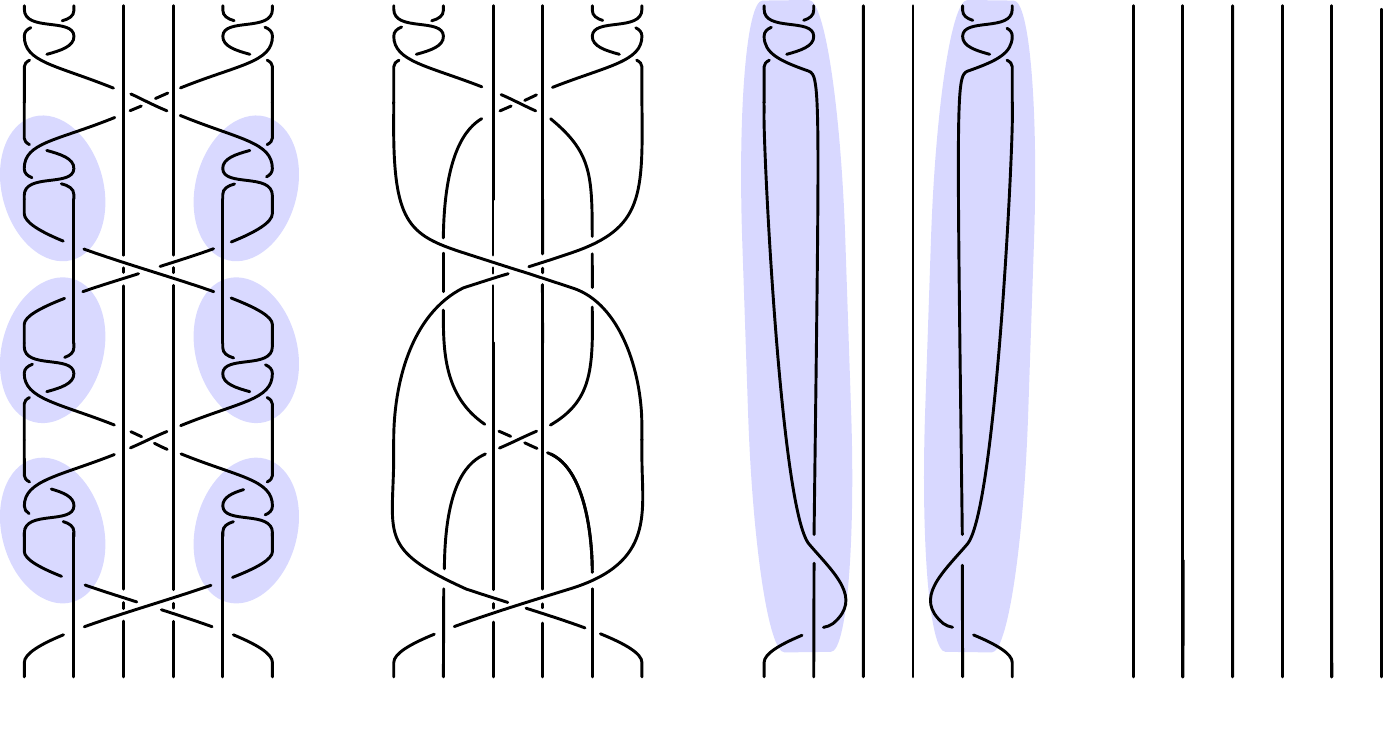
                \caption{Bigelow's 6-strand braid \cite[Figure 4]{Bigelow1999The5} is nontrivial in $B_6$, but it is trivial given the extra relation $\sigma^4=1$. We use this relation to replace $\sigma^3$ with $\sigma^{-1}$ in the shaded areas of Bigelow's braid in diagram (i), isotope diagram (ii), and trivialize instances of $\sigma^4$ in diagram (iii) to obtain a trivial braid (iv).}
                \label{fig:commutator_reduction}
        \end{figure}
    \end{example}

\subsection{Proof of the main theorem.}
    Following \cref{ex:Burau_6}, we have the geometric reasoning we need to prove the theorem under consideration in this section.

    \begin{proof}[Proof of \cref{thm:Burau_roots_of_unity}]
        First we note that one of the containments necessary to establish equation \eqref{eqn:kernel_in_special_cases} can be established by pure computation, which we now show. Again, we let $\sigma \in B_n$ denote one of the half-twist generators of the braid group, $\tau_{n-1} \in B_n$ a full twist on a curve surrounding $n-1$ points in the punctured disk, and $\tau_n \in B_n$ the full twist on the boundary of the disk. One can show that, up to conjugacy, we have
        \begin{center}
            $\beta(\sigma^d) = \begin{pmatrix}
                (-t)^d & 1-t+t^2-\dotsb + (-t)^{d-1} \\ 0 & 1
            \end{pmatrix} \oplus I_{n-3},$ \\
            $\beta(\tau_{n-1}^j) = \widetilde{t^{(n-1)j} \cdot I_{n-2}}, \text{ and } 
            \beta(\tau_n^\ell) = t^{n\ell} \cdot I_{n-1}$
        \end{center}
        where $\widetilde{(-)}:\GL_{n-2}(\mathbb Z[t^\pm]) \to \GL_{n-1}(\mathbb Z(t))$ is the affine extension defined in \cref{def:affine_extension}. Explicitly, we have
        \begin{equation*}
            \beta(\tau_{n-1}^j) = \left( \begin{array}{c|c}
                 t^{(n-1)j} & \frac{1-t^{(n-1)j}}{1-t^{n-1}} \begin{pmatrix} 1-t \\ \vdots \\ 1-t^{n-2} \end{pmatrix}  \\ \hline
                 0 & 1 
            \end{array} \right).
        \end{equation*}
        When $t$ is evaluated at $-q$ for $q \in \mathbb C$ a primitive $d$-th root of unity, it is evident that $\beta(\sigma^d)$ evaluates to the identity. For all of the cases $n,d,j$ indicated in the table of \cref{thm:Burau_roots_of_unity}, with $j < \infty$, one can verify that $-q$ is a $((n-1)j)$-th root of unity and not a $(n-1)$-th root of unity. Thus $\beta(\tau_{n-1}^j)$ evaluates to the identity in all of these cases.
        
        Finally, the matrix $\beta(\tau_n^\ell)$ evaluates to $(-q)^{n\ell} \cdot I_{n-1}$. The values $\ell$ indicated in the table are exactly given as $\ell = 2d/\gcd(2d,(d+2)n)$, which is the order of the unit complex number $(-q)^n$ in the multiplicative group of unit complex numbers. So $\beta(\tau_n^\ell)$ evaluates to the identity for all of the cases given in the table. Thus
        \begin{equation*}
            \ncl_{B_n}(\sigma^d,\tau_{n-1}^j) \cdot \langle \tau_n^\ell \rangle \leq \ker(\beta(-q)).
        \end{equation*}

        For the reverse inclusion, we appeal to the geometry of the moduli space of polyhedra and use the same strategy as \cref{ex:Burau_6}. The appendix of \cite{Thurston1998ShapesSphere} points us to several choices of curvatures $\vec k$ for which the completed moduli space $\overline{\mathcal M}(\vec k)$ is an orbifold. In the cases $\vec k = (k_1,\dotsc,k_n,k_{n+1})$ where $k_1=\dotsb=k_n=k_*$, \cref{cor:commutative_diagram_for_specialization} comes into play to give a restriction on $\ker(\beta(-q))$. Here we give our table again with an extra row giving the number of the relevant entry in Thurston's appendix.
        \begin{center}\begin{tabular}{c||cccccccc|cccc|cc|cc|c|c|c}
                $n$ & \multicolumn{8}{c|}{$4$}                                    & \multicolumn{4}{c|}{$5$}          & \multicolumn{2}{c|}{$6$} & \multicolumn{2}{c|}{$7$}  & $8$ & $9$ & $10$ \\ \hline
                $d$ & $5$ & $6$ & $7$ & $8$ & $9$ & $10$ & $12$ & $18$ & $4$ & $5$ & $6$ & $8$ & $4$ & $5$ & $3$ & $4$ & $3$ & $3$ & $3$  \\
                $j$ & $\infty$ & $\infty$ & $14$ & $8$ & $6$ & $5$  & $4$  & $3$  & $\infty$ & $5$ & $3$ & $2$ & $4$  & $2$ & $\infty$ & $2$ & $6$ & $3$ & $2$  \\
                $l$ & $5$ & $3$ & $7$ & $4$ & $9$ & $5$ & $3$ & $9$ & $4$ & $2$ & $3$ & $8$ & $2$ & $5$ & $6$ & $4$ & $3$ & $2$ & $3$ \\ \hline
                \# & $55.$ & $2.$ & $77.$ & $47.$ & $84.$ & $9.$ & $75.$ & $48.$ & $5.$ & $54.$ & $1.$ & $44.$ & $4.$ & $50.$ & $16.$ & $3.$ & $14.$ & $12.$ & $11.$
        \end{tabular}\end{center}

        For example, entry number 44 in Thurston's appendix gives $\vec k = 2\pi(3,3,3,3,3,1)/8$. One added orbifold stratum in the moduli space, corresponding to a half twist $\sigma$ on two points of the same curvature, has cone angle $\pi-2\pi(3/8) = 2\pi/8$. The other added orbifold stratum, corresponding to a full twist $\tau$ on two points of distinct curvature, has cone angle $2\pi-2\pi(3/8+1/8) = 2\pi/2$. So by \cref{lemma:completion_kernel,lemma:orbifolds_faithful}, the kernel of $\rho_{\vec k}:\Mod(S_{0,6};\vec k) \to \PU(1,3)$ equals $\ncl_{\Mod(S_{0,6};\vec k)}(\sigma^8, \tau^2)$. Since $\sigma$ lifts to a half twist generator $\sigma \in B_5$ and $\tau$ lifts to a full twist on 4 strands $\tau_4 \in B_5$, \cref{cor:commutative_diagram_for_specialization} gives
        \begin{equation*}
            \ker (\beta(-q)) \leq \ncl_{B_n}(\sigma^8,\tau_4^2) \cdot \langle \tau_5 \rangle
        \end{equation*}
        when $q=\exp(i(\pi-2\pi(3/8))) = \exp(2\pi i/8)$ is a primitive 8-th root of unity.
        
        Since conjugates of $\sigma^8$ and $\tau_4^2$ are already in the kernel of $\beta(-q)$ by computation and $\ell=2d/\gcd(2d,(d+2)n) = 8$ is the order of $(-q)^8$, which is the smallest power of $\tau_5$ in the kernel of $\beta(-q)$, we get the slightly stronger restriction
        \begin{equation*}
            \ker (\beta(-q)) \leq \ncl_{B_n}(\sigma^8,\tau_4^2) \cdot \langle \tau_5^8 \rangle.
        \end{equation*}
        This establishes the equality in \cref{thm:Burau_roots_of_unity}. The other cases with $j<\infty$ are similar.

        For the four cases in the table with $j=\infty$, one finds that $k_n+k_{n+1} \geq 2\pi$. In these cases, the completion $\overline{\mathcal M}(\vec k)$ has no stratum added corresponding to the mapping class $\tau$ and so $\ker (\rho_{\vec k})$ is just given as $\ncl_{\Mod(S_{0,n+1};\vec k)}(\sigma^d)$. The rest of the proof is the same.
    \end{proof}
    \begin{remark}
        For three of the four cases in the table with $j=\infty$, the image of $\tau_{n-1} \in B_n$ under $\beta(-q)$ in fact has infinite order. In other words, $\tau_{n-1}^j \notin \ker(\beta(-q))$ for any $j$. However, in the case $n=4$, $d=5$ we see that $\beta(-q)(\tau_3)$ has order 10. Our results thus imply that $\tau_3^{10} \in \ncl_{B_4}(\sigma^5,\tau_4^5 )$. This aligns with \cite[\S 11]{Coxeter1959FactorGroup}, which establishes the order of the central element in the group $B_3/\ncl(\sigma^5)$ (among similar results).
    \end{remark}
    \begin{remark}
        For the cases $(n,d)=(5,6),(7,4)$, and $(4,10)$, the relevant choice of curvatures for the above proof has all curvatures equal: $k_1=\dotsb = k_n=k_{n+1}$. In these cases, the relevant moduli space is a \textit{finite cover} of the moduli space considered by Thurston, the cover in which the $(n+1)$-st point is distinguished. One can check that in these cases the full twist on the $n$-th point and $(n+1)$-st point still corresponds to an orbifold stratum rather than just a cone stratum.
        
        In the same vein, the observant reader will find that we have neglected to use Thurston's entry \#10 which has $\vec k = 2\pi(1,1,\dotsc,1)/6$ with 12 cone points of equal curvature. This corresponds to the case $(n,d)=(11,3)$ in our context. However here, the cover of the moduli space that distinguishes the 12-th point is no longer an orbifold. The stratum corresponding to the twist $\tau$ has cone angle $2\pi-2\pi(1/6+1/6) = 2\pi (2/3)$. So we cannot say for certain what the kernel of the representation $\rho_{\vec k}$ is in this case.
    \end{remark}
    
    Thinking of the broader context here, recall that the $n=4$ case is the last remaining case in which we do not know whether the Burau representation is faithful. \Cref{thm:Burau_roots_of_unity} gives us the restriction on $\ker (\beta_4)$ of \cref{cor:Burau_4_kernel}.

    \begin{proof}[Proof of \cref{cor:Burau_4_kernel}.]
        Any element in the kernel of $\beta_4$, before specialization, is an element of the kernel of \textit{every} specialization. So $\ker (\beta_4)\leq \ker (\beta(-q))$ for any $q$. In particular, the eight entries of \cref{thm:Burau_roots_of_unity} with $n=4$ give restrictions on $\ker (\beta_4)$.

        For the statement about infinite index, note that the quotient $B_4/\ncl(\sigma^d,\tau_3^j,\tau_4)$ is the orbifold fundamental group of some $\overline{\mathcal M}(\vec k)$ in the cases of interest. Part of \cite[Theorem 0.2]{Thurston1998ShapesSphere} is that $\overline{\mathcal M}(\vec k)$ is a complex hyperbolic orbifold of \textit{finite volume}. Therefore, the orbifold fundamental group of this space must be infinite and so the map $B_4 \twoheadrightarrow \Mod(S_{0,5};\vec k) \twoheadrightarrow \pi_1^{orb}(\overline{\mathcal M}(\vec k))$ has a kernel of infinite index. The normal subgroup $\ncl_{B_4}(\sigma^d,\tau_3^j) \cdot \langle \tau_4^\ell \rangle$, with a higher power of $\tau_4$, is a subgroup of this kernel.
    \end{proof}

\subsection{Burau$_3$ at roots of unity}\label{subsec:Burau_3}
    Thurston's appendix does not list choices of curvatures $\vec k$ with 4 cone points because there are infinitely many such choices for which the completion $\overline{\mathcal M}(\vec k)$ is a $\mathbb{CH}^1 \approx \mathbb{RH}^2$ orbifold. In fact, there are infinitely many such $\vec k = (k_1,k_2,k_3,k_4)$ satisfying $k_1=k_2=k_3$, inviting us to include a 3-strand braid group. Using the same method as in the proof of \cref{thm:Burau_roots_of_unity}, we have the following.
    \begin{theorem}\label{thm:Burau_3_at_roots_of_unity}
        Let $q$ be a primitive $d$-th root of unity, $d \geq 7$. Then $$\ker(\beta_3(-q)) = \ncl_{B_3}(\sigma^d) \cdot \langle \tau_3^\ell \rangle$$ where $\ell = 2d/\gcd(12,d+6)$ is the order of the unit complex number $(-q)^3$.
    \end{theorem}
    \begin{proof}
        Take $k_1=k_2=k_3=\pi-2\pi/d$ and $k_4=\pi + 3\pi/d$. With $\vec k = (k_1,k_2,k_3,k_4)$, the same argument as in the proof of \cref{thm:Burau_roots_of_unity} gives the result. Namely, the completion of the moudli space $\overline{\mathcal M}(\vec k)$ has one orbifold stratum added, corresponding to a half twist $\sigma$, of cone angle $\pi-(\pi-2\pi/d) = 2\pi/d$. Note that in this case one has $k_3+k_4 > 2\pi$, so there is no added stratum corresponding to the mapping class $\tau$.
    \end{proof}
    
    This theorem corrects the result stated in \cite[Theorem 1.2]{Funar2014OnUnity}, which is incorrect in the cases when $d$ is a multiple of 3. I would like to note that the correct result was evident in the course of the arguments of that paper. Additionally, when stated in this form, it is clear that \cref{thm:Burau_3_at_roots_of_unity} combines with their result \cite[Theorem 2.1]{Funar2014OnUnity} to give a new proof of the faithfulness of $\beta_3$.
    
\section{Future Directions}\label{sec:conclusion}
    My original motivation for developing these results was to use Thurston's moduli space of polyhedra to place restrictions on the kernel of the Burau representation. Perhaps the restrictions we have found in \cref{cor:Burau_4_kernel} can be combined with other results about the kernel of $\beta_4$ to prove faithfulness of the representation, or to narrow down a search for a nontrivial element of the kernel. The connections with Squier's conjectures yield other questions and future directions.
    
    \subsection*{A mismatch between our results and Squier's conjecture.} Squier conjectured that the kernel of $\beta(-q)$ would be to equal $\ncl(\sigma^d)$ for any $n$ and for $q$ any $d$-th root of unity \cite[(C1)]{Squier1984TheUnitary}. From the results of \cite{Funar2014OnUnity} and the typical nonfaithfulness of Burau \cite{Bigelow1999The5}, we know that for any $n \geq 5$ this conjecture is false for almost all even values of $d$. The form we have acquired in \cref{thm:Burau_roots_of_unity} for the finitely many special cases of $\ker(\beta(-q))$ certainly looks more complicated than Squier's very simple normal generating set. It would be interesting to see in which cases the normal subgroup $\ncl(\sigma^d,\tau_{n-1}^j) \cdot \langle \tau_n^\ell \rangle$ is or is not equal to $\ncl(\sigma^d)$, i.e. to see when Squier's conjecture is correct as originally stated. Part of investigating this question could be a calculation along the lines of \cref{fig:commutator_reduction}.

    \subsection*{The case with more than \texorpdfstring{$n+1$}{n+1} cone points.} One restriction in this work is that we have only considered choices of curvatures $\vec k = (k_1,\dotsc,k_m)$ with $m=n+1$. There are two reasons for this. First, taking $m=n+1$ allows us to identify the kernel of the composite map
    \begin{equation*}
        B_n \xtwoheadrightarrow{\iota_*} \Mod(S_{0,n+1};\vec k) \xrightarrow{\rho_{\vec k}} \PU(1,n-2)
    \end{equation*}
    in terms of $\ker (\rho_{\vec k})$ because the map of mapping class groups $\iota_*$ is surjective. When $m > n+1$, the induced map on mapping class groups is injective (and not surjective) \cite[Theorem 3.18]{Farb2012AGroups}:
    \begin{equation*}
        B_n \xhookrightarrow{\iota_*} \Mod(S_{0,m};\vec k) \xrightarrow{\rho_{\vec k}} \PU(1,m-3)
    \end{equation*}
    but I was unable to rigorously identify the preimage of $\ker (\rho_{\vec k})$ in $B_n$ in this case. I would like to see how to identify the kernel in this case and whether it could give any more information about the Burau representation.

    Second, the condition $m=n+1$ tells us that the ``evaluation map'' in \cref{thm:commutative_diagram} is actually an evaluation, and so we can place the evaluation of the Burau representation in the commutative diagram as in \cref{cor:commutative_diagram_for_specialization}. When $m>n+1$, this is not the case. It is not hard to find braids $b \in B_n$ and roots of unity $q$ for which $\beta(b)|_{t=-q}$ is trivial while $\ev(-q)(\beta(b))$ is not. For instance, powers of the full twist braid exhibit this behavior. So it is not immediately clear to me how the moduli space of polyhedra could be used to place restrictions on $\ker(\beta(-q))$ aside from the cases explored in this paper.

    \subsection*{Similar results for the Gassner representation.} To get further mileage out of Thurston's list, one might consider the Gassner representation of the \textit{pure} braid group
    $$\mathcal{G}_n:PB_n \to \operatorname{GL}_{n-1}(\mathbb{Z}[t_1^\pm,\dotsc,t_n^\pm]),$$
    which specializes to the Burau representation by sending $t_i \mapsto t$ for all $i$. See \cite[Chapter 3]{Birman1975BraidsAM-82} for a construction. Since the Gassner representation is defined on a subgroup of the braid group and takes values in a larger matrix group, one might expect it to be ``more faithful'' than the Burau representation. Yet it is not known whether or not $\mathcal{G}_n$ is faithful for any value of $n \geq 4$. A version of  \cref{thm:commutative_diagram} for the Gassner representation exists, for instance in \cite{Venkataramana2014MonodromyLine} via an algebraic approach. A geometric or topological construction would be desirable along the lines of our \cref{thm:commutative_diagram}, though our method of proof would be much more cumbersome in this context due to the more complicated generating set of the pure braid group. Yet armed with a version of \cref{thm:commutative_diagram} and \cref{cor:commutative_diagram_for_specialization} for the Gassner representation, more of Thurston's 94 moduli spaces could help to identify the kernels of some specializations of various $\mathcal G_n$. Though finitely many restrictions on $\ker \mathcal G_n$ could never establish faithfulness alone (see \cite[Lemma 2.1]{Long1986AGroups}), perhaps this could shed some light on the faithfulness question for these representations.

    \subsection*{Relationship with some remarkable work of Coxeter.} In \cite{Coxeter1959FactorGroup}, Coxeter investigated the quotients of braid groups defined by $B_n(d) = B_n/\ncl(\sigma^d)$. For $n=2$, the quotient is a finite cyclic group. For $d=2$, the quotient is isomorphic to the symmetric group on $n$ letters. For all but five other choices of $(n,d)$, the quotient $B_n(d)$ is an infinite group. Coxeter established infiniteness in these cases using hyperbolic geometry, and I think this is along the lines of the argument for infinite index in the proof of \cref{cor:Burau_4_kernel}. The five sporadic cases of $(n,d)$ for which $B_n(d)$ is finite correspond to the Platonic solids, and Coxeter gave a remarkable formula for the order of $B_n(d)$ in terms of the combinatorics of the associated Platonic solid. Coxeter proved his formula by individually computing the orders of these five quotient groups and checking that the formula works in each case. I hope that the geometric perspective of the moduli spaces considered in this paper might give more insight into -- and perhaps a more revealing proof of -- Coxeter's result.

\bibliographystyle{amsalpha}    
\bibliography{references.bib}
\end{document}

%% file: Diagrams/developing.pdf_tex
\begingroup%
  \makeatletter%
  \providecommand\color[2][]{%
    \errmessage{(Inkscape) Color is used for the text in Inkscape, but the package 'color.sty' is not loaded}%
    \renewcommand\color[2][]{}%
  }%
  \providecommand\transparent[1]{%
    \errmessage{(Inkscape) Transparency is used (non-zero) for the text in Inkscape, but the package 'transparent.sty' is not loaded}%
    \renewcommand\transparent[1]{}%
  }%
  \providecommand\rotatebox[2]{#2}%
  \newcommand*\fsize{\dimexpr\f@size pt\relax}%
  \newcommand*\lineheight[1]{\fontsize{\fsize}{#1\fsize}\selectfont}%
  \ifx\svgwidth\undefined%
    \setlength{\unitlength}{602.94750784bp}%
    \ifx\svgscale\undefined%
      \relax%
    \else%
      \setlength{\unitlength}{\unitlength * \real{\svgscale}}%
    \fi%
  \else%
    \setlength{\unitlength}{\svgwidth}%
  \fi%
  \global\let\svgwidth\undefined%
  \global\let\svgscale\undefined%
  \makeatother%
  \begin{picture}(1,0.42875744)%
    \lineheight{1}%
    \setlength\tabcolsep{0pt}%
    \put(0,0){\includegraphics[width=\unitlength,page=1]{developing.pdf}}%
    \put(0.18889418,0.00385637){\color[rgb]{0,0,0}\makebox(0,0)[lt]{\lineheight{1.25}\smash{\begin{tabular}[t]{l}$b_m$\end{tabular}}}}%
    \put(0.04759445,0.39445093){\color[rgb]{0,0,0}\makebox(0,0)[lt]{\lineheight{1.25}\smash{\begin{tabular}[t]{l}$b_1$\end{tabular}}}}%
    \put(0.16107642,0.41506139){\color[rgb]{0,0,0}\makebox(0,0)[lt]{\lineheight{1.25}\smash{\begin{tabular}[t]{l}$b_{m-1}$\end{tabular}}}}%
    \put(0.17035505,0.19357753){\color[rgb]{0,0,0}\makebox(0,0)[lt]{\lineheight{1.25}\smash{\begin{tabular}[t]{l}$b_i$\end{tabular}}}}%
    \put(0.03376722,0.20098612){\color[rgb]{0,0,0}\makebox(0,0)[lt]{\lineheight{1.25}\smash{\begin{tabular}[t]{l}$b_{i-1}$\end{tabular}}}}%
    \put(0.30789213,0.2033735){\color[rgb]{0,0,0}\makebox(0,0)[lt]{\lineheight{1.25}\smash{\begin{tabular}[t]{l}$b_{i+1}$\end{tabular}}}}%
    \put(0,0){\includegraphics[width=\unitlength,page=2]{developing.pdf}}%
    \put(0.64299365,0.06691801){\color[rgb]{0,0,0}\makebox(0,0)[lt]{\lineheight{1.25}\smash{\begin{tabular}[t]{l}$z_{i-1}$\end{tabular}}}}%
    \put(0.7839056,0.06821698){\color[rgb]{0,0,0}\makebox(0,0)[lt]{\lineheight{1.25}\smash{\begin{tabular}[t]{l}$z_i$\end{tabular}}}}%
    \put(0.91108126,0.1108178){\color[rgb]{0,0,0}\makebox(0,0)[lt]{\lineheight{1.25}\smash{\begin{tabular}[t]{l}$z_{i+1}$\end{tabular}}}}%
    \put(0.79087974,0.34191403){\color[rgb]{0,0,0}\makebox(0,0)[lt]{\lineheight{1.25}\smash{\begin{tabular}[t]{l}$z_{m-2}$\end{tabular}}}}%
    \put(0.64499216,0.3588623){\color[rgb]{0,0,0}\makebox(0,0)[lt]{\lineheight{1.25}\smash{\begin{tabular}[t]{l}$z_{m-1}$\end{tabular}}}}%
    \put(0.55614137,0.25233798){\color[rgb]{0,0,0}\makebox(0,0)[lt]{\lineheight{1.25}\smash{\begin{tabular}[t]{l}$z_1$\end{tabular}}}}%
    \put(0,0){\includegraphics[width=\unitlength,page=3]{developing.pdf}}%
    \put(0.47437417,0.36190886){\color[rgb]{0,0,0}\makebox(0,0)[lt]{\lineheight{0}\smash{\begin{tabular}[t]{l}$d$\end{tabular}}}}%
    \put(0,0){\includegraphics[width=\unitlength,page=4]{developing.pdf}}%
  \end{picture}%
\endgroup%

%% file: Diagrams/braid_action.pdf_tex
\begingroup%
  \makeatletter%
  \providecommand\color[2][]{%
    \errmessage{(Inkscape) Color is used for the text in Inkscape, but the package 'color.sty' is not loaded}%
    \renewcommand\color[2][]{}%
  }%
  \providecommand\transparent[1]{%
    \errmessage{(Inkscape) Transparency is used (non-zero) for the text in Inkscape, but the package 'transparent.sty' is not loaded}%
    \renewcommand\transparent[1]{}%
  }%
  \providecommand\rotatebox[2]{#2}%
  \newcommand*\fsize{\dimexpr\f@size pt\relax}%
  \newcommand*\lineheight[1]{\fontsize{\fsize}{#1\fsize}\selectfont}%
  \ifx\svgwidth\undefined%
    \setlength{\unitlength}{572.02954102bp}%
    \ifx\svgscale\undefined%
      \relax%
    \else%
      \setlength{\unitlength}{\unitlength * \real{\svgscale}}%
    \fi%
  \else%
    \setlength{\unitlength}{\svgwidth}%
  \fi%
  \global\let\svgwidth\undefined%
  \global\let\svgscale\undefined%
  \makeatother%
  \begin{picture}(1,0.75690769)%
    \lineheight{1}%
    \setlength\tabcolsep{0pt}%
    \put(0,0){\includegraphics[width=\unitlength,page=1]{braid_action.pdf}}%
    \put(0.18148692,0.52351401){\color[rgb]{0,0,0}\makebox(0,0)[lt]{\lineheight{1.25}\smash{\begin{tabular}[t]{l}$b_i$\end{tabular}}}}%
    \put(0.32012377,0.53325965){\color[rgb]{0,0,0}\makebox(0,0)[lt]{\lineheight{1.25}\smash{\begin{tabular}[t]{l}$b_{i+1}$\end{tabular}}}}%
    \put(0,0){\includegraphics[width=\unitlength,page=2]{braid_action.pdf}}%
    \put(0.07673642,0.08709781){\color[rgb]{0,0,0}\makebox(0,0)[lt]{\lineheight{1.25}\smash{\begin{tabular}[t]{l}$z_{i-1}$\end{tabular}}}}%
    \put(0.22788684,0.08846702){\color[rgb]{0,0,0}\makebox(0,0)[lt]{\lineheight{1.25}\smash{\begin{tabular}[t]{l}$z_i$\end{tabular}}}}%
    \put(0.35931404,0.13599256){\color[rgb]{0,0,0}\makebox(0,0)[lt]{\lineheight{1.25}\smash{\begin{tabular}[t]{l}$z_{i+1}$\end{tabular}}}}%
    \put(0,0){\includegraphics[width=\unitlength,page=3]{braid_action.pdf}}%
    \put(0.80996914,0.0863673){\color[rgb]{0,0,0}\makebox(0,0)[lt]{\lineheight{1.25}\smash{\begin{tabular}[t]{l}$k_*$\end{tabular}}}}%
    \put(0,0){\includegraphics[width=\unitlength,page=4]{braid_action.pdf}}%
    \put(0.44637749,0.61388743){\color[rgb]{0,0,0}\makebox(0,0)[lt]{\lineheight{1.25}\smash{\begin{tabular}[t]{l}$\iota_*(\sigma_i)$\end{tabular}}}}%
    \put(0,0){\includegraphics[width=\unitlength,page=5]{braid_action.pdf}}%
    \put(0.71786676,0.56553811){\color[rgb]{0,0,0}\makebox(0,0)[lt]{\lineheight{1.25}\smash{\begin{tabular}[t]{l}$b_{i+1}$\end{tabular}}}}%
    \put(0.91874403,0.53588189){\color[rgb]{0,0,0}\makebox(0,0)[lt]{\lineheight{1.25}\smash{\begin{tabular}[t]{l}$b_i$\end{tabular}}}}%
    \put(0,0){\includegraphics[width=\unitlength,page=6]{braid_action.pdf}}%
    \put(0.17702027,0.42800546){\color[rgb]{0,0,0}\makebox(0,0)[lt]{\lineheight{0}\smash{\begin{tabular}[t]{l}$d$\end{tabular}}}}%
    \put(0,0){\includegraphics[width=\unitlength,page=7]{braid_action.pdf}}%
    \put(0.7771249,0.42805748){\color[rgb]{0,0,0}\makebox(0,0)[lt]{\lineheight{0}\smash{\begin{tabular}[t]{l}$d$\end{tabular}}}}%
    \put(0,0){\includegraphics[width=\unitlength,page=8]{braid_action.pdf}}%
  \end{picture}%
\endgroup%

%% file: Diagrams/the_last_twist_in_the_disk.pdf_tex
\begingroup%
  \makeatletter%
  \providecommand\color[2][]{%
    \errmessage{(Inkscape) Color is used for the text in Inkscape, but the package 'color.sty' is not loaded}%
    \renewcommand\color[2][]{}%
  }%
  \providecommand\transparent[1]{%
    \errmessage{(Inkscape) Transparency is used (non-zero) for the text in Inkscape, but the package 'transparent.sty' is not loaded}%
    \renewcommand\transparent[1]{}%
  }%
  \providecommand\rotatebox[2]{#2}%
  \newcommand*\fsize{\dimexpr\f@size pt\relax}%
  \newcommand*\lineheight[1]{\fontsize{\fsize}{#1\fsize}\selectfont}%
  \ifx\svgwidth\undefined%
    \setlength{\unitlength}{154.86684964bp}%
    \ifx\svgscale\undefined%
      \relax%
    \else%
      \setlength{\unitlength}{\unitlength * \real{\svgscale}}%
    \fi%
  \else%
    \setlength{\unitlength}{\svgwidth}%
  \fi%
  \global\let\svgwidth\undefined%
  \global\let\svgscale\undefined%
  \makeatother%
  \begin{picture}(1,0.984792)%
    \lineheight{1}%
    \setlength\tabcolsep{0pt}%
    \put(0,0){\includegraphics[width=\unitlength,page=1]{the_last_twist_in_the_disk.pdf}}%
  \end{picture}%
\endgroup%

%% file: Diagrams/bigelow_braids_manipulate.pdf_tex
\begingroup%
  \makeatletter%
  \providecommand\color[2][]{%
    \errmessage{(Inkscape) Color is used for the text in Inkscape, but the package 'color.sty' is not loaded}%
    \renewcommand\color[2][]{}%
  }%
  \providecommand\transparent[1]{%
    \errmessage{(Inkscape) Transparency is used (non-zero) for the text in Inkscape, but the package 'transparent.sty' is not loaded}%
    \renewcommand\transparent[1]{}%
  }%
  \providecommand\rotatebox[2]{#2}%
  \newcommand*\fsize{\dimexpr\f@size pt\relax}%
  \newcommand*\lineheight[1]{\fontsize{\fsize}{#1\fsize}\selectfont}%
  \ifx\svgwidth\undefined%
    \setlength{\unitlength}{398.42361832bp}%
    \ifx\svgscale\undefined%
      \relax%
    \else%
      \setlength{\unitlength}{\unitlength * \real{\svgscale}}%
    \fi%
  \else%
    \setlength{\unitlength}{\svgwidth}%
  \fi%
  \global\let\svgwidth\undefined%
  \global\let\svgscale\undefined%
  \makeatother%
  \begin{picture}(1,0.52777092)%
    \lineheight{1}%
    \setlength\tabcolsep{0pt}%
    \put(0,0){\includegraphics[width=\unitlength,page=1]{bigelow_braids_manipulate.pdf}}%
    \put(0.09325896,0.00477959){\color[rgb]{0.03921569,0,0}\transparent{0.996503}\makebox(0,0)[lt]{\lineheight{1.25}\smash{\begin{tabular}[t]{l}(i)\end{tabular}}}}%
    \put(0.35845677,0.00477959){\color[rgb]{0.03921569,0,0}\transparent{0.996503}\makebox(0,0)[lt]{\lineheight{1.25}\smash{\begin{tabular}[t]{l}(ii)\end{tabular}}}}%
    \put(0.62119711,0.00477959){\color[rgb]{0.03921569,0,0}\transparent{0.996503}\makebox(0,0)[lt]{\lineheight{1.25}\smash{\begin{tabular}[t]{l}(iii)\end{tabular}}}}%
    \put(0.88710846,0.00477959){\color[rgb]{0.03921569,0,0}\transparent{0.996503}\makebox(0,0)[lt]{\lineheight{1.25}\smash{\begin{tabular}[t]{l}(iv)\end{tabular}}}}%
  \end{picture}%
\endgroup%